\documentclass[12pt, reqno]{amsart}
\usepackage{amssymb,amsmath,amsthm,newlfont,color}
  \usepackage{psfrag}
\usepackage[T1]{fontenc}
\usepackage{a4wide}
\usepackage[dvipsnames]{xcolor}
\usepackage{amssymb}
\usepackage{amsmath}
\usepackage{amsthm}
\usepackage[all]{xy}
\usepackage{accents}
\usepackage{comment}
\usepackage{eurosym}
\usepackage{graphicx}  
\usepackage[bottom]{footmisc}
\usepackage[colorlinks=true,linkcolor=blue,urlcolor=blue,citecolor=blue]{hyperref}
\usepackage{nomencl}
\usepackage{glossaries}
\usepackage{tikz}
\usepackage[abs]{overpic}
\usepackage{pict2e}
\usepackage{vmargin}
\usepackage{enumitem}
\usepackage{multicol}

\newcommand\C{\mathbb{C}}

\newcommand\N{\mathbb{N}}

\newcommand{\dist}{\operatorname{dist}}
\newcommand{\re}{\operatorname{Re}}
\newcommand{\im}{\operatorname{Im}}

\newtheorem{theorem}{Theorem}[section]
\newtheorem*{theorem*}{Theorem} 
\newtheorem{lemma}[theorem]{Lemma}

\newtheorem{remark}{Remark}

\begin{document}

\title{Zero and Uniqueness Sets for Fock spaces}

\author[D. Aadi]{D. Aadi}
\address{D. AADI, Mohammed V University in Rabat, Faculty of sciences, CeReMAR -LAMA- B.P. 1014 Rabat, Morocco.}
\email{driss\_aadi@um5.ac.ma \ -- \ aadidriss@gmail.com }

\author[Y. Omari]{Y. OMARI}
\address{Y. Omari, Mohammed V University in Rabat, Faculty of sciences, CeReMAR -LAMA- B.P. 1014 Rabat, Morocco.}
\email{omariysf@gmail.com}

\date{\today}

\maketitle

\begin{abstract}
We characterize zero sets for which every subset remains a zero set too in the Fock space $\mathcal{F}^p$, $1\leq p<\infty$. We are also interested in the study of a stability problem for some examples of uniqueness set with zero excess in Fock spaces.
\end{abstract}

\section{Introduction}

Let $\beta$ be a positive real number.  The Gaussian measure on the complex plane $\mathbb{C}$  is defined by
\begin{equation}\label{gausian}
  d\mu_{\beta}(z):=\frac{\beta}{2\pi}e^{-\frac{\beta}{2}|z|^2}dA(z),  \qquad z\in\mathbb{C},
\end{equation}
where $dA$ is the Euclidean area measure. The Fock space $\mathcal{F}_{}^p:=\mathcal{F}_{\pi}^p$,  where $1\leq p<\infty$,  is the collection of entire functions $f : \mathbb{C} \rightarrow \mathbb{C}$ such that
\[   \|f\|_{p}:= \left( \int_{\mathbb{C}}\left|f(z)\right|^p d\mu_{p\pi}(z)\right)^{\frac{1}{p}}<\infty.  \]
The space $\mathcal{F}^p$ endowed with the norm $\|.\|_{p}$ is a vector Banach space, for every $p\geq1$. For the particular case when $p=2$, $\mathcal{F}_{}^2$ is a reproducing kernel Hilbert space with the reproducing kernel given by
$$K(z,w):=e^{\pi\Bar{z}w},\quad z,w\in\C.$$
For instance see the textbook \cite[Chapitre 2]{Zhu} and references therein.  \\

A countable set $Z=\{z_{n}\}_{n\in\mathbb{N}}\subseteq\C$ is called a zero set for $\mathcal{F}_{}^p$ if there exists a function $f\in\mathcal{F}_{}^p\setminus\{0\}$ such that the zero set $\{z\in\C\ :\ f(z)=0\}$ of $f,$ counting multiplicities, coincides with $Z$. We say that  $Z$ is a uniqueness set for $\mathcal{F}_{}^p$ if the unique function of $\mathcal{F}^p$ that vanishes on $Z$ is the zero function.
 It is known that a complete characterization, of zero and uniqueness sets for the Fock spaces, stills remain an open question, we refer to \cite{Zhu,zhu2011maximal,aadi2018zero}.\\
 
Due to the distinctiveness of Fock spaces among other spaces of analytic functions, there exist particular sets; uniqueness set with zero excess.  These are those uniqueness sets that they become zero sets by removing just one point.  Numerous examples of uniqueness sets with zero excess are known for $\mathcal{F}^p$. The first typical example is the square lattice for $\mathcal{F}^{p}$, when $2<p<\infty$, and the square lattice without one point for every $1<p\leq 2$. More generally, for every $0\leq \nu \leq 1$, the sequence $\Gamma_\nu$ is a uniqueness set with zero excess for $\mathcal{F}^p$, for every $\frac{2}{1+\nu}< p < \frac{2}{\nu}$, where
 \begin{equation}
 \Gamma_{\nu}:=\left\{\gamma_{m,n}:=m+in\ :\ (m,n)\in \mathcal{I}_\nu \right\},
\end{equation}
and where $\mathcal{I}_\nu:=\left(\mathbb{Z}\times(\mathbb{Z}\setminus\{0\})\right)\cup \left(\mathbb{Z}_{-}\times\{0\}\right)\cup\{(m+\nu,0)\}_{m\geq0}$, see \cite{lyubarskii1992frames,zhu2011maximal}. It is simple to see that $\Gamma_\nu$ is a separated sequence in the complex plan and of critical density for $\mathcal{F}_{}^p$ in the study of interpolating and sampling problems, see for example \cite{Se92,zhu2011maximal,Zhu}. Motivated by the results in \cite{omari2020stability,zhu2011maximal,horowitz1974zeros}, we are interested in the study of a stability problem of sequences $\Gamma_\nu$, for every $0\leq\nu\leq 1$.
\\

The second kind of uniqueness sets with zero excess for $\mathcal{F}_{}^p$ is an example of a non separated sequence localized at the real and imaginary lines. It constitutes the zero set of the sin-cardinal type function 
\begin{equation}
    S(z):=(z^2-1)\frac{\sin\big(\frac{\pi}{2}z^2 \big)}{\pi z^2},\quad z\in\C,
\end{equation}
and it is given by
\begin{equation}\label{AsLyuSe}
\Gamma:=\left\{\pm\sqrt{2n},\ \pm i\sqrt{2n} \ :\ n\geq 1\right\}\cup\left\{\pm 1\right\}.
\end{equation}
This sequence was constructed by Ascensi, Lyubarski and Seip in \cite{ascensi2009phase} for the Hilbert case. It stills of the same kind for $\mathcal{F}_{}^p$, for every $1<p<\infty$.\\

As in the work of the second author in \cite{omari2020stability}, we are interested in the study of a stability problem of the sequences $\Gamma$ and $\Gamma_\nu$, for every $0\leq\nu\leq 1$, for the spaces $\mathcal{F}^p$, for fixed $1<p<\infty$. Namely, If $\Lambda=\{\lambda_\sigma\ :\ \sigma\in\Sigma\}$ is a sequence of complex numbers, where $\Sigma=\Gamma$ or $\Gamma_\nu$. We write $\lambda_\sigma=\sigma e^{\delta_\sigma}e^{i\theta_\sigma}$, where $\delta_\sigma, \theta_\sigma\in\mathbb{R}$ for every $\sigma\in\Sigma$. We are interesting in giving optimal conditions on $(\delta_\sigma)$ and $(\theta_\sigma)$ for which $\Lambda$ fails to be a uniqueness set with zero excess for $\mathcal{F}^p$, for a fixed $1<p<\infty$. Before stating our main results, we need some notations. For $\Lambda:=\{\lambda_\gamma:=\gamma e^{\delta_\gamma}e^{i\theta_\gamma}\ :\ \gamma\in\Gamma_\nu\}$, we denote
$$\hat{\delta}(\Lambda):=\liminf_{R\rightarrow \infty} \frac{1}{\log R} \sum_{|\gamma|\leq R}\delta_\gamma\ \ \mbox{and}\ \ \delta(\Lambda):=\limsup_{R\rightarrow \infty} \frac{1}{\log R} \sum_{|\gamma|\leq R}\delta_\gamma.$$
Our first main result in this paper is the following theorem.
\begin{theorem}\label{thm1}
Let $0\leq \nu\leq 1$ and let $\Lambda=\{\lambda_\gamma\}_{\gamma\in\Gamma_\nu}$ be a sequence of complex numbers. We write $\lambda_\gamma = \gamma e^{\delta_\gamma}e^{i\theta_\gamma}$, for every $\gamma\in\Gamma_\nu$, where $\delta_\gamma,\theta_\gamma\in\mathbb{R}$. If 
\begin{enumerate}
    \item $\Lambda$ is separated,
    \item The sequences $(\gamma^2\delta_\gamma)_{\gamma\in\Gamma_\nu}$ and $(\gamma^2\theta_\gamma)_{\gamma\in\Gamma_\nu}$ are bounded,
    \item $$\nu-\frac{2}{p}<\hat{\delta}(\Lambda) \leq \delta(\Lambda) < \nu+1-\frac{2}{p}.$$
\end{enumerate}
Then, $\Lambda$ is a uniqueness set with zero excess for $\mathcal{F}^p$, whenever $\frac{2}{1+\nu}<p < \frac{2}{\nu}$.
\end{theorem}

Now, for the Ascensi-Lyubarskii-Seip sequence $\Gamma$. If $\Lambda:=\{\lambda_\gamma:=\gamma e^{\delta_\gamma}e^{i\theta_\gamma}\ :\ \gamma\in\Gamma\}$ is a sequence of $\C$. We will denote by $\Delta_n$ the quantity
$$\Delta_n:=\sum_{|\gamma|=\sqrt{2n}} \delta_\gamma=\delta_{\sqrt{2n}}+\delta_{-\sqrt{2n}}+\delta_{i\sqrt{2n}}+\delta_{-i\sqrt{2n}}.$$
Our second main result is the following.
\begin{theorem}\label{thm2}
Let $\Gamma:=\left\{\pm\sqrt{2n},\ \pm i\sqrt{2n} \ :\ n\geq 1\right\}\cup\{\pm 1\}$ and let $\Lambda=\{\lambda_\gamma:=\gamma e^{\delta_\gamma}e^{i\theta_\gamma}\ :\ \gamma\in\Gamma\}$ be a sequence of complex numbers. Suppose that
\begin{enumerate}
    \item There exists $c>0$ such that $\left|\lambda_\gamma - \lambda_{\gamma'}\right| \geq c/\min\{|\gamma|,|\gamma'|\}$, for every $\gamma, \gamma' \in\Gamma$,
    \item The sequences $(\gamma^2\delta_\gamma)_{\gamma\in\Gamma}$ and $(\gamma^2\theta_\gamma)_{\gamma\in\Gamma}$ are bounded,
    \item $\Delta(\Lambda) := \underset{n\rightarrow\infty}{\limsup}\  \frac{1}{\log n} \left|\underset{k=1}{\overset{n}{\sum}} \Delta_k\right|< \frac{1}{2\max\{p,q\}}$, where $q$ is the H\"older conjugate number of $p$.
\end{enumerate}
Then, $\Lambda$ is a uniqueness set with a  zero excess for $\mathcal{F}^p$, where $1<p<\infty$.
\end{theorem}

Another extreme case of sequences, we are interesting in, is motivated by the zero set of the sin-cardinal type function 
\[ s(z)=\frac{\sin\big(\frac{\pi}{2} z^2\big)}{z^2} \in\mathcal{F}^{p}.\]
The zero set of $s$, denoted by $Z(s)$, is a zero sequence for $\mathcal{F}^p$, for every $1\leq p<\infty$. However, if we remove the subset which belongs to the imaginary axis from  $Z(s)$, the remaining part is not a zero set anymore for $\mathcal{F}^{p}$. Such result can be viewed as a consequence of the  Lindel\"{o}f's theorem, see \cite[Theorem 2.10.1]{Boa}. Therefore, a natural question is: which zero set for $\mathcal{F}^{p}$ remains a zero set too for $\mathcal{F}^{p}$, even an infinite subset was removed? \\

Actually, the example above is a variant to the one given by Zhu in \cite{zhu1993zeros}. This phenomena is one of the main deference between Fock spaces and Hardy spaces and even Bergman spaces of the unit disk where zero sets are well stable \cite{duren1970theory,hedenmalm2012theory}.\\

In the following theorem, we give a complete description of zero sets for which all their sub-sequences are also zero sets for $\mathcal{F}^p.$
\begin{theorem}\label{thm3}
Let $Z =\{z_n\}_{n\in\N}$ be a zero set for $\mathcal{F}^p$, $1\leq p<\infty$. The following statements are equivalents
\begin{enumerate}
    \item Every subset of $Z$ is a zero set for $\mathcal{F}^p$,
    \item  $Z$ satisfies  
    \begin{equation}\label{SufficientCond}
        \underset{n\in\mathbb{N}}{\sum}\frac{1}{|z_n|^2} < \infty.
    \end{equation}
\end{enumerate}
\end{theorem}
Before stating the proofs of our main results, we give first some remarks:
\begin{enumerate}

\item In Theorem \ref{thm1}, if $\nu=1$ we then get $1<p<2$. Actually, the result remains valid for $p=2$ and this case was treated in \cite{omari2020stability}.  Theorem \ref{thm1} gives a result analogous to those related to complete interpolating sequences for the Paley-Wiener spaces, see \cite{avdonin1974question,lyubarskii1997complete}, and small Fock spaces \cite{baranov2015sampling,omari2021complete}.
\item Note that if there  exists a positive integer $N$ such that 
\begin{equation}
    \underset{n\geq 0}{\sup}\ \frac{n+1}{N}\left|\underset{k=n+1}{\overset{n+N}{\sum}}\Delta_k\right|<\frac{1}{2\max\{p,q\}}\label{avd}
\end{equation}
then $\underset{N\rightarrow\infty}{\lim} \frac{1}{\log N}\left|\underset{k=0}{\overset{N}{\sum}}\Delta_k\right|<\frac{1}{2\max\{p,q\}}$ and the converse is not true, see \cite[Lemma 5.4]{omari2020stability}. On the other hand, in Theorem \ref{thm2}, for $p=2$ we have $q=2$. This case was treated in \cite{omari2020stability}. Theorem \ref{thm2} with the Avdonin's type condition \eqref{avd} appears like the result proved by Lyubarskii and Seip in \cite{lyubarskii1997complete} concerning complete interpolating sequences for Paley-Wiener spaces. Such result generalizes those by Kadet and Avdonin for the Hilbert case, see \cite{kadets1964exact,avdonin1974question}.
\item The conditions on the sequences $(\delta_\gamma)$ and $(\theta_\gamma)$ in Theorem \ref{thm1} are optimal. The proof is similar to theorem 1.5  and proposition 5.3 in \cite{omari2020stability}.
\item  An interesting fact appears in the proof of Theorem \ref{thm2} (namely Lemma \ref{lem2}) is a confirmation of the result which confirms that $\mathcal{F}^{p}$ and $\mathcal{F}^{q}$, $p>q$ do not share the same zero sets, a result that we have already got in \cite{aadi2018zero}. We provided a sequence with positive Beurling uniform density, while the example that we can construct here, by a precise choice of $\delta$, is of null lower Beurling density.
\end{enumerate}

We end this section with some words on notation: throughout this paper, the notation $A\lesssim B$ means that $A\leq cB$ for a certain positive constant $c$, and the notation $A\asymp B$ will be used to say that $A\lesssim B$ and $B\lesssim A$ hold in the same time. The paper is organized as follows: In the next we state some key Lemmas containing estimates of some modified infinite products. Section \ref{Secproof} is devoted to prove Theorems \ref{thm1} and \ref{thm2}. Theorem \ref{thm3} will be proved in the last section. 

\section{Some lemmas}
In this section, we introduce some modified Weierstrass products. These functions will play an important role in the proof of our main results. First, we recall that for every $0\leq \nu\leq 1$, the sequence $\Gamma_\nu=\{\gamma_{m,n}\ :\ m,n\in\mathbb{Z}\}$ is given by 
$$\left\{\gamma_{m,n}:=m+in\ :\ (m,n)\in\mathbb{Z}\times\left(\mathbb{Z}\setminus\{0\}\right) \right\}\cup \mathbb{Z}_{-}\cup\{m+\nu\}_{m\geq 0}.$$
If $\Lambda=\{\lambda_{m,n}\ :\ m,n\in\mathbb{Z}\}$ is a sequence of complex numbers. We will write $\lambda_{m,n}:=\gamma_{m,n} e^{\delta_{m,n}}e^{i\theta_{m,n}}$, where $\delta_{m,n}, \theta_{m,n} \in\mathbb{R}$, for every $m,n\in\mathbb{Z}$. We associate with $\Lambda$ the following infinite product
    $$G_\Lambda(z) := (z-\lambda_{0,0})\prod_{m,n\in\mathbb{Z}}\ ^{\prime} \left(1-\frac{z}{\lambda_{m,n}}\right)\exp\left[\frac{z}{\gamma_{m,n}}+\frac{z^2}{2\gamma_{m,n}^2}\right],\quad z\in\C.$$
The product with the prime is taken over all integers $m$ and $n$ such that $(m, n) \neq (0, 0)$. The following lemma provides an estimates of the function $G_\Lambda$.
\begin{lemma}[\cite{omari2020stability}, Lemma 3.2]\label{lem1} If $\Lambda$ satisfies the conditions of Theorem \ref{thm1}, $G_\Lambda$ is an entire function vanishing exactly on $\Lambda$ and verifying $$\frac{(1+|\im(z)|)^M}{(1+|z|)^{\nu-\hat{\delta}+M}} \dist(z,\Lambda) \lesssim |G_\Lambda(z)|e^{-\frac{\pi}{2}|z|^2} \lesssim \frac{(1+|z|)^{-\nu+\delta+M}}{(1+|\im(z)|)^M} \dist(z,\Lambda), \ z\in\mathbb{C},$$
for some positive constant $M$, where $\delta=\delta(\Lambda)+\varepsilon$ and $\hat{\delta}=\hat{\delta}(\Lambda)-\varepsilon$ for a positive $\varepsilon$ small enough.
\end{lemma}

On the other hand, for the Ascensi-Lyubarskii-Seip sequence given in \cite{ascensi2009phase} by
$$\Gamma =\left\{\pm\sqrt{2n},\ \pm i\sqrt{2n}\ :\ n\geq 1\right\}\cup\{\pm 1\},$$
we associate the modified sin-cardinal function  $G_\Gamma(z):=\frac{z^2-1}{\pi z^2}\sin\left(\frac{\pi}{2}z^2\right)$. Now, if $\Lambda=\{\lambda_\gamma:=\gamma e^{\delta_\gamma}e^{i\theta_\gamma}\ :\ \gamma\in\Gamma \}$ is a sequence of complex numbers, we associate with $\Lambda$ the infinite product 
$$G_\Lambda(z):=\lim_{r\rightarrow\infty}\prod_{\lambda\in\lambda,\ |\lambda|\leq r} \left(1-\frac{z}{\lambda}\right),\quad z\in\C.$$
According to the proof of \cite[Theorem 1.10]{omari2020stability}, we have the following lemma that gives an estimate of the function $G_\Lambda$.

\begin{lemma}\label{lem2}
If $\Lambda$ satisfies the conditions of Theorem \ref{thm2}, the infinite product $G_\Lambda$ converges uniformly on every compact set of $\C$ and verifies:
\[
\frac{\dist(z,\Lambda)}{\dist(z,\Gamma)}\frac{(1+|\im z^2|)^M}{(1+|z|)^{2\delta+2M}}|G_\Gamma(z)|\lesssim |G_\Lambda(z)|,
\]
\[
|G_\Lambda(z)| \lesssim \frac{\dist(z,\Lambda)}{\dist(z,\Gamma)}\frac{(1+|z|)^{2\delta+2M}}{(1+|\im z^2|)^M}|G_\Gamma(z)|,
\]
for every $z\in\C\setminus\Gamma$, where $\delta=\Delta(\Lambda)+\varepsilon$ for a small positive $\varepsilon.$
\end{lemma}

\section{Proofs of Theorems \ref{thm1} and \ref{thm2}}\label{Secproof}
This section is devoted to the proofs of Theorem \ref{thm1} and \ref{thm2}.

\subsection*{Proof of Theorem \ref{thm1}.}
First, we show that $\Lambda\setminus\{\lambda\}$ is a zero set for $\mathcal{F}^p$, for some fixed (any) $\lambda\in\Lambda$. To this end, it suffices to prove that  $\frac{G_\Lambda}{z-\lambda}$ belongs to $\mathcal{F}^p$. Indeed, by Lemma \ref{lem1} we have
\begin{eqnarray*}
\int_\C \left|\frac{G_\Lambda(z)}{z-\lambda}e^{-\frac{\pi}{2}|z|^2}\right|^p dA(z) & \lesssim & \int_\C \frac{(1+|z|)^{p(-1-\nu+\delta+M)}}{(1+|\im(z)|)^{pM}} dA(z)\\
   & \asymp & \int_\C \frac{1}{(1+|z|)^{p(1+\nu-\delta)}}dA(z),
\end{eqnarray*}
the last integral converges if and only if $p(1+\nu-\delta)>2$. In view of the third assumption, the integral converges. Hence,  $\frac{G_\Lambda}{z-\lambda} \in\mathcal{F}^{p}$ (obviously $z\mapsto \frac{G_\Lambda(z)}{z-\lambda}$ is an entire function).  
\\
Secondly, we prove that $\Lambda$ is a uniqueness set for $\mathcal{F}^p$. Let $F$ be a function of $\mathcal{F}^p$ that vanishes on $\Lambda$. Then there exists an entire function $h$  such that $F=hG_\Lambda$. According to the estimates of $G_\Lambda$ in Lemma \ref{lem1}, we have 
\begin{eqnarray*}
|h(z)|\frac{(1+|\im(z)|)^M}{(1+|z|)^{\nu-\hat{\delta}+M}} \dist(z,\Lambda) & \lesssim & \left|h(z)G_\Lambda(z)\right|e^{-\frac{\pi}{2}|z|^2} = |F(z)|e^{-\frac{\pi}{2}|z|^2} \lesssim 1.
\end{eqnarray*}
This implies that $h$ is a polynomial of $z$, we denote later by $k$ its degree. Integrating the last inequality with respect to the measure $dA(z)$, we get 

\begin{eqnarray*}
\int_\C \left|F(z)e^{-\frac{\pi}{2}|z|^2} \right|^p dA(z) & \gtrsim & \int_\C |h(z)|^p\frac{(1+|\im(z)|)^{pM}}{(1+|z|)^{p(\nu-\hat{\delta}+M)}} \dist(z,\Lambda)^pdA(z)\\
     & \asymp & \int_\C \frac{(1+|\im(z)|)^{pM}}{(1+|z|)^{p(\nu-k-\hat{\delta}+M)}}dA(z) \\
     & \asymp & \int_\C \frac{1}{(1+|z|)^{p(\nu-k-\hat{\delta})}}dA(z) .
\end{eqnarray*}
The last integral converges if and only if $p(\nu-k-\hat{\delta})>2$ and this implies that $k+\hat{\delta}<\nu-2/p$. This is in contradiction with the assumption $\nu-2/p < \hat{\delta}(\Lambda)$. Thus $h$ is zero and $F$ too. This completes the proof of Theorem \ref{thm1}.

\subsection*{Proof of Theorem \ref{thm2}.}
 First, we need the lemma below which is analogous to Lemma 3.4 in \cite{omari2020stability}, we include the proof for completeness. We denote by 
\[d\nu_{p,\alpha,\beta}(z)=\Big(\frac{1+|z|^2}{1+|\im z^2|}\Big)^{\alpha p}\frac{1}{(1+|z|)^{p\beta}} e^{-\frac{p\pi}{2}|z|^2}dA(z),\]
where $1\leq p<\infty$ and $\alpha$ and $\beta$ are two real numbers. 
\begin{lemma}\label{lem3}
Let $\alpha$ and $\beta$ two real numbers. The sin-cardinal type function $G_\Gamma$ 
belongs to $L^p(\C,d\nu_{p,\alpha,\beta})$ if and only if $\beta>\frac{1}{p}$.
\end{lemma}

\begin{proof}
Let $1\leq p<\infty$. Recall that,
\[ \left| \sin(z)\right|^2=   \left(\sin\big(\re z ) \right)^2  +\left(\sinh(\im z)  \right)^2, \qquad z\in\mathbb{C}.\]
It follows that,  $G_{\Gamma}$ belongs to  $L^{p}(\mathbb{C}, d\nu_{p,\alpha,\beta})$ if and only if $\sinh\left(\frac{\pi}{2}z^2\right)$ does. This is equivalent to 
\[ \int_{\mathbb{C}}  \left(\frac{e^{\frac{\pi}{2} |\im z^2|}  }{1+|z|^{\beta}}\right)^{p} \left(\frac{1+|z|^2}{1+|\im z^2|}\right)^{p \alpha }   e^{-\frac{p \pi}{2}|z|^{2}}dA(z)<\infty. \]
Using Tonelli theorem,  we obtain
\begin{align}\label{integI}
   I & := \int_{|z|>1}  \frac{1}{(1+|z|)^{p\beta-2p\alpha}} \frac{1}{(1+|\im z^2|)^{p\alpha}}  e^{-\frac{p\pi}{2} (|z|^2-|\im z^2|)}dA(z)\nonumber  \\
      &\asymp 8 \int_{1}^{\infty} \int_{0}^{x} \frac{e^{-\frac{p\pi}{2} (x-y)^{2} }  } {(x^2+y^2)^{p(\beta-2\alpha)/2}} \frac{1}{(1+xy)^{p \alpha} } dy dx.   
\end{align}
On the other hand, we have
\begin{align}\label{changeI}
\int_{0}^{x} \frac{e^{-\frac{p\pi}{2} (x-y)^{2} }} {(x^2+y^2)^{p(\beta-2\alpha)/2}} \frac{1}{(1+xy)^{p \alpha}} dy  &   
\asymp \frac{1}{x^{p(\beta-2\alpha)}}\int_{0}^{\infty} \frac{e^{-\frac{p\pi}{2} y^{2} }}{  (1+x^2-xy)^{p \alpha}}\chi_{[0,x]}(y) dy.
\end{align}
Combining \eqref{integI} and \eqref{changeI},  we obtain
\begin{align*}
   I & \asymp \int_1^\infty \int_0^\infty   \frac{1}{x^{p(\beta-2\alpha)}}   \frac{e^{-\frac{p\pi}{2} y^2}}{  (1+x^2-xy)^{p\alpha}  }\chi_{[0,x]}(y) dx dy\\
   & = \int_{0}^{\infty}  e^{-\frac{p\pi}{2} y^{2} } \int_{\max\{1,y\}}^{\infty} \frac{1}{ x^{p(\beta-2\alpha)}(1+x^2-xy)^{p \alpha} } dxdy.
\end{align*}
Consequently, the integral needed converges if and only if 
 $$\int_{\max\{1,y\}}^{\infty} \frac{1}{  x^{p(\beta-2\alpha)}(1+x^2-xy)^{p \alpha} } dx $$
converges too. That is,  if and only if, $\beta>\frac{1}{p}$. 
\end{proof}

Now, we can start the proof Theorem \ref{thm2}. First, we will show that $\Lambda\setminus\{\lambda\}$ is a zero set for $\mathcal{F}^p$, for fixed $\lambda\in\Lambda$ and every $1<p<\infty$. To do this, it suffices to prove that $\frac{G_\Lambda}{z-\lambda}$ belongs to $\mathcal{F}^p$. According to Lemma \ref{lem2} and by a sub-harmonicity argument, we have 
\begin{align*}
\int_\C\left|\frac{G_\Lambda(z)}{z-\lambda}e^{-\frac{\pi}{2}|z|^2}\right|^p dA(z) & \asymp \int_\C\left|\frac{G_\Lambda(z)}{z-\lambda}\frac{\dist(z,\Gamma)}{\dist(z,\Lambda)}e^{-\frac{\pi}{2}|z|^2}\right|^p dA(z) \\
      & \lesssim \int_\C |G_\Gamma(z)|^p\left(\frac{1+|z|^2}{1+|\im z^2|}\right)^{pM}\frac{e^{-\frac{p\pi}{2}|z|^2}}{(1+|z|)^{p(1-2\delta)}}dA(z)\\
      & =  \int_\C |G_\Gamma(z)|^pd\nu_{p,M,1-2\delta}(z).
\end{align*}
By Lemma \ref{lem3}, the last integral converges since $2\delta<1-\frac{1}{p}=\frac{1}{q}.$\\

To prove that $\Lambda$ is a uniqueness set for $\mathcal{F}^p$, $1<p<\infty$, let $F$ be a function of $\mathcal{F}^p$ that vanishes on $\Lambda$. Write $F(z)=h(z)G_\Lambda(z)$, for some entire function $h$. Again, by Lemma \ref{lem2} we have

\begin{align}\label{est1}
|F(z)| \frac{\dist(z,\Gamma)}{\dist(z,\Lambda)} 
\gtrsim |h(z)G_\Gamma(z)|  \frac{(1+|\im z^2|)^M}{(1+|z|)^{2\delta+2M}}.
\end{align}
Integrating both sides of the last inequality over $\C$  with respect to the measure $d\mu_{p\pi}(z)=e^{-\frac{p\pi}{2}|z|^2}dA(z)$
\begin{align*}
\int_\C\left|F(z)\right|^p d\mu_{p\pi}(z) & \asymp \int_\C\left|F(z)\frac{\dist(z,\Gamma)}{\dist(z,\Lambda)} \right|^p d\mu_{p\pi}(z) \\
     & \gtrsim \int_\C \left|h(z)G_\Gamma(z) \frac{(1+|\im z^2|)^M}{(1+|z|)^{2\delta+2M}}\right|^p d\mu_{p\pi}(z)\\
      & \gtrsim \int_\C \left| \frac{h(z)G_\Gamma(z)}{(1+|z|)^{2\delta+2M}}\right|^p d\mu_{p\pi}(z).
\end{align*}
In the second line, we have used a sub-harmonicity argument. Since $2\delta< \frac{1}{\max\{p,q\}}\leq \frac{1}{2}$, then for a fixed $\gamma\in\Gamma$, we have
\begin{align*}
  \int_\C \left| \frac{h(z)G_\Gamma(z)}{(1+|z|)^{2\delta+2M}}\right|^p d\mu_{p\pi}(z) &  \gtrsim \int_\C \left| \frac{h(z)G_\Gamma(z)}{(z-\gamma)P_{2M}(z)}\right|^p d\mu_{p\pi}(z),
\end{align*}
where $P_{2M}$ is  a polynomial of degree $\lfloor2M\rfloor+1$, that vanishes on $\lfloor2M\rfloor+1$ points on $\Gamma\setminus\{\gamma\}$. This implies that the function $z\mapsto \frac{h(z)G_\Gamma(z)}{(z-\gamma)P_{2M}(z)}$ belongs to $\mathcal{F}^p$. Since the sequence $\Gamma\setminus\{\gamma\}$ is a maximal zero sequence for $\mathcal{F}^p$, then $h$ must be a polynomial of degree less than $\lfloor2M\rfloor+1$ (see \cite{Zhu,zhu2011maximal}). Suppose that $h$ is not zero and denote $k$ its degree. Now, return to \eqref{est1} and integrate both sides over $\C$ with respect to the measure $d\mu_{p\pi}$ again, we obtain

\begin{align*}
\int_\C|F(z)|^p d\mu_{p\pi}(z) & \asymp \int_C\left|F(z)\frac{\dist(z,\Gamma)}{\dist(z,\Lambda)} \right|^p d\mu_{p\pi}(z) \\
     & \gtrsim \int_\C \left|h(z)G_\Gamma(z) \frac{(1+|\im z^2|)^M}{(1+|z|)^{2\delta+2M}}\right|^p d\mu_{p\pi}(z)\\
     & \asymp \int_\C \left|G_\Gamma(z) \frac{(1+|\im z^2|)^M}{(1+|z|)^{2\delta+2M-k}}\right|^p d\mu_{p\pi}(z)\\
     & = \int_\C|G_\Gamma(z)|^p d\nu_{p,-M,2\delta-k}(z).
\end{align*}
By Lemma \ref{lem3}, the latter integral converges if and only if $2\delta-k>1/p$. Since $\delta<1/(2p)$, we then get $$1/p<2\delta-k<2\delta< 1/p.$$
This is a contradiction. Hence, $h$ and $F$ are zero. Therefore, $\Lambda$ is a uniqueness set for $\mathcal{F}^p$. This completes the proof of Theorem   \ref{thm2}.

\section{Proof of Theorem \ref{thm3}}

The proof of Theorem \ref{thm3} is essentially based on  Lindel\"{o}f's theorem below.  First, we recall some main tools  very useful for our proof, we refer to \cite{Boa,Lev} for more details. If $f$ is an entire function and $r$ a positive real number, we denote $ M(r,f)$ the maximum modulus of $f$ on the circle $|z|=r$ \[M(r,f)=\underset{|z|=r}{\max}|f(z)|.\] The order of $f$ is given by the quantity
\[\rho_{f}:=\underset{r\rightarrow\infty}{\limsup}\ \frac{\log\log M(r,f)}{\log(r)}.\]
Always, we have $0\leq \rho_{f}\leq\infty$. In the case $0<\rho_{f} <\infty$, the type of $f$ is defined by $$\tau_f:= \underset{r\rightarrow\infty}{\limsup}\ \frac{\log M(r,f)}{r^{\rho_{f}}}.$$ 
Let $Z=\{z_n\}_{n\in\mathbb{N}}$ be the zero set of an entire function $f$. Following \cite{Boa}, the convergence exponent of the sequence $Z=\{z_n\}_{n\in\mathbb{N}}$ (excluding $0$ if it belongs to $Z$) is defined as the infimum of all positive numbers $s$ such that 
\begin{equation}
    \sum_{n\in\mathbb{N}} \frac{1}{|z_n|^s}<\infty,
\end{equation}
it will be denoted by $\rho_1$ (for short, the convergence exponent of $f$).
 A consequence of Jensen's formula gives the following relations among the order $\rho_f$ and the exponent of convergence $\rho_1$ of an entire function $f$  (see, \cite{Boa} for complete proof):
\begin{equation}\label{ExponentAndorder}
    \rho_1\leq \rho_f.
\end{equation}
The following theorem characterizes entire functions of integral order and of finite type.
\begin{theorem}[Lindel\"{o}f, \cite{Boa}]\label{Lindlof}
If $\rho$ is a positive integer, the entire function $f$ of order $\rho_f=\rho$ is of finite type if and only if 
\begin{enumerate}
\item $n(r)=O(r^{\rho_f})$, where $n(r)$ is the number of zeros of $f$ in the disk $|z|\leq r$, counting multiplicity,
\item and the sums 
\begin{equation}
S(r) := \sum_{|z_n|\leq r}\frac{1}{z_n^\rho}
\end{equation}
are bounded,  where $\{z_n\}_{n}$ is the zero sequence of $f$.
\end{enumerate}
\end{theorem}

In the sequel, our constructions are based on dividing the complex plane into sectors with some defined opening aperture. To this end, for a given two angles $\beta\in(-\pi,\pi]$ and $\theta\in(0,\pi]$ define  
\[ \mathcal{S}(\beta,\theta):=\{z\in\C\ :\ |\arg(z)-\beta|\leq\theta\}\cup\{z\in\C\ :\ |\arg(-z)-\beta|\leq\theta \}.\]
The following lemma will be of prominent role in the proof of Theorem \ref{thm3} later on.
\begin{lemma} \label{lemmaCN}
Let $1\leq p<\infty$. If $Z=\{z_n\}_{n\in\mathbb{N}}$ is a zero set for $\mathcal{F}^{p}_{}$ such that $Z\subset \mathcal{S}(\beta,\theta)$ for some $0\leq\theta<\frac{\pi}{4}$ and $\beta\in(-\pi,\pi]$, then $Z$ satisfies 
\begin{equation} \label{SC}
\sum_{n\in\mathbb{N}} \frac{1}{|z_n|^{2}}<\infty.
\end{equation}
\end{lemma}
\begin{proof}
Without loss of generality, we may suppose  $\beta=0$. Aiming to come to a contradiction, assume that 
\begin{equation}\label{Dv}
\underset{n\in\mathbb{N}}{\overset{}{\sum}}\ \frac{1}{|z_n|^2}\ =\ \infty.
\end{equation}
 If $g$ is a function in  $\mathcal{F}^{p}_{}$ with $Z(g)=Z$, then by  \cite[Theorem 5.1]{Zhu}, for every $\epsilon>0$ we have 
\begin{equation}\label{NC}
\sum_{n\in\mathbb{N}} \frac{1}{|z_n|^{2+\epsilon}}<\infty.
\end{equation}
Thus, $Z=\{z_n\}_{n\in\mathbb{N}}$ is of convergence exponent $\rho_1=2$.
Combining \eqref{ExponentAndorder}, \eqref{NC} and \eqref{Dv} we obtain 
\[ 2=\rho_1\leq \rho_g\leq 2. \]
Hence,  $g$ is of order  $2$,  and  of type $\tau_{g}$ less than or equal to $\frac{\pi}{2}$. Since $g$ is of integral order, Lindel\"{o}f's 
Theorem applies.
 Writing $z_n=|z_n|e^{i\theta_n}$, $n\in\mathbb{N}$, a straightforward calculation gives
\begin{eqnarray*}
|S(r)|^2 & = & \left|\underset{|z_n|\leq r}{\sum} \frac{1}{z_n^2}\right|^2=\left|\underset{|z_n|\leq r}{\sum} \frac{e^{-2i\theta_n}}{|z_n|^2}\right|^2 \\
  & = & \left|\underset{|z_n|\leq r}{\sum} \frac{\cos\big(2\theta_n\big)-i\sin\big(2\theta_n\big)}{|z_n|^2}\right|^2\\
   & = & \left(\underset{|z_n|\leq r}{\sum} \frac{\cos\big(2\theta_n\big)}{|z_n|^2}\right)^2 + \left(\underset{|z_n|\leq r}{\sum} \frac{\sin\big(2\theta_n\big)}{|z_n|^2}\right)^2 \\
   & \geq & \left(\underset{|z_n|\leq r}{\sum} \frac{\cos\big(2\theta_n\big)}{|z_n|^2}\right)^2 \\
    & \geq & \cos^2(2\theta)\left(\underset{|z_n|\leq r}{\sum} \frac{1}{|z_n|^2}\right)^2 \longrightarrow\infty, \qquad \mbox{as} \ r\rightarrow \infty.
\end{eqnarray*}
This contradicts Lindel\"{o}f's Theorem \ref{Lindlof}.  As a conclusion
\begin{equation*}
\underset{n=0}{\overset{\infty}{\sum}}\ \frac{1}{|z_n|^2}\ < \infty.
\end{equation*}
\end{proof}
\begin{remark}
We mention that the constant  $\frac{\pi}{4}$, that appears in  Lemma \ref{lemmaCN}, is the best possible. A counterexample is given by the zero set of the sin-cardinal function
\[ s(z)=\frac{\sin(\frac{\pi}{2}z^2)}{z^2}\in\mathcal{F}^{p}.\]
\end{remark}
Now, we can prove Theorem \ref{thm3}.\

\subsection*{Proof of Theorem \ref{thm3}.} 
Let $Z=\{z_n\}_{n\in\mathbb{N}}$ be a sequence of complex numbers which is a zero sequence for $\mathcal{F}^{p}_{}$ and satisfies \eqref{SufficientCond}. If $Z'$ is any sub-sequence of $Z$, then it is a zero sequence too for $\mathcal{F}^{p}_{}$ by the sufficient condition \cite[Theorem 5.3]{Zhu}.  Therefore, $Z$ belongs to the desired class.\\
Conversely,  let $Z$ be a zero sequence for $\mathcal{F}^{p}_{}$ with the property: every subset of $Z$ is also a zero set for $\mathcal{F}^p$.  Dividing $Z$ into eight subset, by writing 
\[\displaystyle Z=\bigcup_{k=0}^{7}Z_k,\]
where for each $0\leq k\leq 7$ 
\[Z_{k}:=Z\cap\{z\in\C\ :\ -\frac{\pi}{8}\leq \mathrm{arg}(z)-\frac{k\pi}{4}<\frac{\pi}{8}\}\subset \mathcal{S}(\frac{\pi}{8}, \frac{\pi}{8}+\varepsilon),\]
and $\varepsilon$ is an arbitrary small positive number in $(0,\frac{\pi}{8})$. By the assumption, each $Z_{k}$,  $k\in\{0,1,\cdots,7\}$ is a zero set for $\mathcal{F}^{p}$. We conclude by Lemma \ref{lemmaCN}. 

\subsection*{Acknowledgments}
The authors are deeply grateful to Professor O. El-Fallah (Mohammed V university, Faculty of Sciences in Rabat)  for many helpful discussions, suggestions and remarks.


\end{document}